\documentclass[10pt]{amsart}
\usepackage{amssymb, amsthm}\usepackage{amsmath}\usepackage{epsfig}
\usepackage{fancyhdr}

\theoremstyle{plain}
\newtheorem{Lem}{Lemma}[section]

\newtheorem{Prop}[Lem]{Proposition}

\newtheorem{The}[Lem]{Theorem}

\theoremstyle{definition}

\newtheorem*{Que}{Question}

\newtheorem{definition}[Lem]{Definition}
\newtheorem{Rem}[Lem]{Remark}
\newtheorem{Exe}[Lem]{Example}

\newcommand{\fvie}[1]{{\rm Stab}(#1)}
\def\smap{\mathcal{F}}

\def\wsmap{\smap^*}
\def\msmap{\smap^\star}
\def\Fix{\mathop{\rm Fix}}
\def\Orb{\mathop{\rm Orb}}
\def\Stab{\mathop{\rm Stab}}
\def\Atrac{\mathop{\rm Atrac}}

\begin{document}
\author{Godelle Eddy}
\title{Stable set of self-map}
\maketitle
\begin{abstract} The attracting set and the inverse limit set are important objects associated to a self-map on a set. We call \emph{stable set} of the self-map the projection of the inverse limit set. It is included in the attracting set, but is not equal in the general case. Here  we determine whether or not the equality holds in several particular cases, among which are the case of a dense range continuous function on an Hilbert space, and the case of a substitution over left infinite words.    

keywords : self-map, orbit set, stable set, attracting set.\end{abstract}
\section{Introduction}
Let~$X$ be a set and $\varphi: X\to X$ be a self-map. One can associated to~$\varphi$ four subsets of $X$, namely its \emph{fixed set}~$\Fix(\varphi)$, its \emph{orbit set}~$\Orb(\varphi)$, its \emph{stable set}~$\fvie{\varphi}$ and its \emph{attracting set}~$\Atrac(\varphi)$, which are defined by 
$$\Fix(\varphi) = \{x\in  E\mid \varphi(x) = x\},$$ 
$$\Orb(\varphi) = \cup_{n\in \mathbb{N}}\Fix(\varphi^n),$$ 
$$\fvie{\varphi} = \{x\in X\mid \exists (x_n)_{n\geq 0}, \varphi(x_{n+1}) = x_n\textrm{ and  }x_0 = x\},$$
$$\Atrac(\varphi) = \bigcap_{n\in \mathbb{N}} \varphi^n(X).$$ 
It is clear that the following sequence of inclusions holds. \begin{equation}\Fix(\varphi) \subseteq \Orb(\varphi) \subseteq \Stab(\varphi)\subseteq \Atrac(\varphi).\label{suiteinclusions}\end{equation} 
 These four subsets are the objects of numerous articles. The fixed set of an automorphism of a finitely generated (free) group has been widely studied~\cite{BeHa,DySc,JaSh,Shp,Ven}. The orbit set and the attracting set are considered in  Dynamic System Theory~\cite{ChEvWa,FaJo}. The stable set has been considered in Mathematical Economy~\cite{Hol,MeRa}, in Theoretical Computer Science~\cite{GlLeRi,ShWa}, in analysis~\cite{PoC2} and in Dynamic System Theory~\cite{Arnetall}. Moreover, the stable set is implicitly a key tool in one of the proof of Scott Conjecture~\cite{ImrTur}. In general the inclusions in the sequence~(\ref{suiteinclusions}) are not equalities. In this note we  address the question of deciding whether or not the latter inclusion is indeed an equality under some various extra structural hypotheses. For instance, we consider the case of dense range continuous self-maps on an Hilbert space, homomorphisms of free groups and substitution over right infinite words. Surprisingly, we have not be able to find any reference where this natural question is even mentioned. One of the motivation for this work is to fill this gap. Another motivation is a better understanding of some families of infinite words that occur as stable-set of a family of substitutions ({\it cf.} Section~\ref{sectgenera} for details). The two following theorems gather our main results: 
\begin{The} Assume one of the following cases:\label{THintro2}
\begin{enumerate}
\item the set~$X$ is an Hilbert space with a countably infinite base, and~$\varphi$ is a linear continuous self-map with dense range;
\item the set~$X$ is the open unit disc~$\mathbb{D}$ of the complex plane, and~$\varphi$ is an analytic function such that $|\varphi(z)|<1$.
\end{enumerate}
Then, the equality $\Stab(\varphi) = \Atrac(\varphi)$ may not hold.
\end{The}
\begin{The} Assume one of the following cases:\label{THintro1}
\begin{enumerate}
\item the set~$X$ is a compact metric space, and~$\varphi$ is a continuous self-map;
\item the set~$X$ is a limit group of free groups, and~$\varphi$ is a group endomorphism;
\item the set~$X$ is the set of finite words (or of left infinite words) over a finite set, and $\varphi$ is a substitution.
\end{enumerate}
Then, the equality $\Stab(\varphi) =  \Atrac(\varphi)$ holds.
\end{The}

 The paper is organized as it follows, in~Section~2 we present easy examples where the equality either holds or does not hold. We also explain the connection between the stable set and the inverse limit of a self-map. In Sections~\ref{sectespvect} and~\ref{secteanalyt}, we focus on the cases considered in Theorem~\ref{THintro2}. In Sections~\ref{sectcompa},~\ref{sectfree} and~\ref{sectsubsti}, we turn to the cases considered in Theorem~\ref{THintro1}. We also investigate the question of deciding whether or not the other inclusions of the sequence~(\ref{suiteinclusions}) are actually equalities. Finally, in Section~\ref{sectgenera} we replace the map~$\varphi$ by a monoid generated by a given set of self-maps, and introduce a generalized definition of the notion of a self-map. Then, we provided motivating examples for such a definition. 

\section{Example, counter-example and inverse limit}
\label{section:example}
Let us start with a simple counter example to the equality, and obvious cases where the equality holds.
\begin{Exe} Set~$X = \{(n,m)\in \mathbb{Z}^2 \mid 0\leq m\leq \max(n-1,0)\}$ and let $\varphi : X\to X$ be defined by $\varphi(n,m) = (n,m-1)$ for positive~$m$, and $\varphi(n,0) = (min (n-1,0), 0)$. The stable set of $\varphi$ is empty whereas the attracting set is~$\{(n,0)\mid n\leq 0\}$.  \label{lem:contreexemple}   
 \begin{figure}[ht]
\begin{picture}(200,100)
\put(18,0){\includegraphics[scale = 0.6]{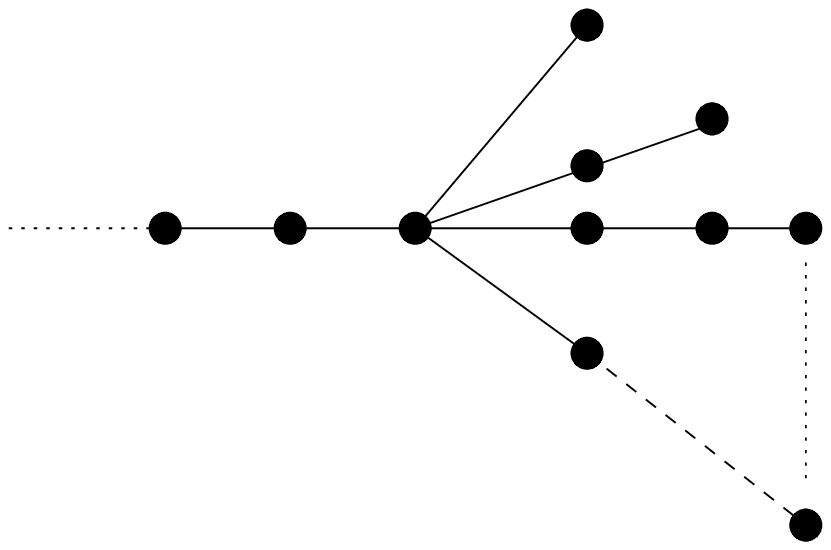}}
\put(30,44){{\small $(-2,0)$}}\put(52,60){{\small $(-1,0)$}}\put(75,43){{\small $(0,0)$}} 
 \put(122,90){{\small $(1,0)$}} \put(109,71){{\small $(2,0)$}} \put(144,73){{\small $(2,1)$}} \put(108,44){{\small $(3,0)$}}   
\put(133,44){{\small $(3,1)$}} \put(162,53){{\small $(3,2)$}} \put(105,20){{\small $(n,0)$}}\put(162,1){{\small $(n,n-1)$}}  
\end{picture}
\caption{A counter example}\label{figurecontreexemple}
\end{figure}
\end{Exe}

The set~$\fvie{\varphi}$ can be characterized as a maximal subset: 
\begin{Lem}Let $X$ be a set and $\varphi : X\to X$ be a self-map. 
 One has $$\varphi(\fvie{\varphi}) = \fvie{\varphi}.$$ Furthermore,~$\fvie{\varphi}$ contains every~subset~$Y$ of $X$ such that~$\varphi(Y) = Y$.\label{lemfac1}
\end{Lem}
This implies that for every self-map~$\psi :X\to X$ such that $\psi(\fvie{\varphi}) = \fvie{\varphi}$, one has $\fvie{\varphi}\subseteq \fvie{\psi}$. In particular, for every positive integer~$n$, one has $\fvie{\varphi^n} = \fvie{\varphi}$.
\begin{Lem} Let $X$ be a set and $\varphi : X\to X$ be a self-map. If $\varphi$ is either injective on some~$\varphi^n(X)$ or surjective, then $\Stab(\varphi) =  \Atrac(\varphi)$. Furthermore, in the former case the restriction~$\varphi : \fvie{\varphi}\to \fvie{\varphi}$ is bijective.\label{lemfac2} 
\end{Lem}
The proof of Lemma~\ref{lemfac1} and~\ref{lemfac2} is obvious, and let to the reader. We end this section by explaining the connection between the inverse limit of a self-map and its stable set. Let us recall the definition of the former.
\begin{definition} Let $P$ be a poset and~$(X_i)_{i\in P}$ be a family of sets. Assume that for each pair $(i,j)$ of $P$ such that $i\leq j$ we have a map~$\varphi_{i,j} : X_j\to X_i$ such that $\varphi_{i,i} = Id_{X_i}$ and $\varphi_{i,j}\circ \varphi_{j,k} = \varphi_{i,k}$  for $i\leq j\leq k$. Then, the inverse limit~$\underleftarrow{\lim}(X_i)$ of the \emph{projective system}~$(X_i, \varphi_{i,j})$ is the set $$\{(x_i)_{i\in P}\mid x_i\in X_i,\ \varphi_{i,j}(x_j) = x_i \}.$$   
\end{definition}
Now, for each index~$j$, we have projection maps~$\pi_j: \underleftarrow{\lim}(X_i)\to X_j, (x_i)_{i\in P}\mapsto x_j$. If we choose $P = \mathbb{N}$, $X_i = X$ and $\varphi_{i,j} = \varphi$ for every~$i,j$, then by definition ~$\fvie{\varphi}$ is equal to~$\pi_0(\underleftarrow{\lim}(X_i))$ ( $= \pi_j(\underleftarrow{\lim}(X_i))$). 

\section{Vector spaces and linear maps}
\label{sectespvect} Here we consider the case of a linear self-map on a vector space. 
One can see without difficulty that the equality $\fvie{\varphi} = \Atrac(\varphi)$ does not hold in general for infinite dimensional vector spaces: Consider Example~\ref{lem:contreexemple} and its notations; let~$V$ be the vector space with base~$X$, and extend~$\varphi$ from~$X$ to~$V$ by linearity. Nevertheless, there is a case where the equality always holds.    
\begin{Prop} \label{Prop:infdim}Let $V$ be an infinite dimensional vector space, and $\varphi: V\to V$ be a linear map. If the sequence $(Ker(\varphi^n))$ eventually stabilizes then $$\fvie{\varphi} = \Atrac(\varphi).$$ 
\end{Prop}
Note that we may have $\fvie{\varphi} = \Atrac(\varphi)$ even if the sequence $(Ker(\varphi^n))$ does not stabilize. The above result applies in particular when $V$ is a finite dimensional vector space. 
\begin{proof}
Assume $Ker(\varphi^{n+1}) = Ker(\varphi^n)$ for some~$n$. Let~$y$ belong to $Im(\varphi^n)$ and be such that $\varphi(y) = 0$. Choose $x$ such that $\varphi^n(x) = y$. Then $\varphi^{n+1}(x) = 0$, and therefore $ y  = \varphi^n(x) = 0$. Hence,~$\varphi$ is injective on~$Im(\varphi^n)$ and we conclude using Lemma~\ref{lemfac2}. 
\end{proof}

Among infinite dimensional vector spaces, the case of Hilbert spaces is of importance. The notion of a linear continuous self-map with dense range on an infinite dimensional Hilbert space is crucial in Operator Theory. If~$H$ is an Hilbert space (with a countably infinite base) and $T: H \to H$ is a linear continuous self-map, then we know by~\cite{Est,Ste} that $\fvie{T}$ is non-empty, since $H$ is a complete metric space. A simple example of such a map is the following: 
  
\begin{Exe} \label{lem:infdim2} Let~$H$ be the Hilbert space~$\ell^2(\mathbb{C})$ and fix an Hilbert base : $(e_k)_{k\geq 0}$. We define $T: H\to H$ by $T(\sum \lambda_k e_k) = \sum\frac{\lambda_{k+1}}{k+1} e_k$. This map is clearly not injective, continuous and compact with dense range. One can see that~$T$ is not surjective since the element~$\sum\frac{1}{k+1} e_k$ does not have a pre-image. Finally, the kernel of~$\varphi$ is generated by~$e_1$, hence, $\Atrac(T) = \fvie{T}$ by Proposition~\ref{Prop:infdimker1} below.
\end{Exe}

\begin{Prop} \label{Prop:infdimker1}Let $V$ be an infinite dimensional vector space, and $\varphi: V\to V$ be a linear map such that $dim(Ker(\varphi)) = 1$. Then, $$\fvie{\varphi} = \Atrac(\varphi).$$ 
\end{Prop}

\begin{proof}
Since $dim(Ker(\varphi)) = 1$, either $Ker(\varphi)$ is included in~$\Atrac(\varphi)$, or $\varphi$ is injective on some $Im(\varphi^n)$. In the latter case, we apply Lemma~\ref{lemfac2}. Assume the former case holds. Then we claim that $\varphi(\Atrac(\varphi)) = \Atrac(\varphi)$. Clearly, $\varphi(\Atrac(\varphi))$ is included in~$\Atrac(\varphi)$. Conversely, assume~$\varphi(x)$ belongs to~$\Atrac(\varphi)$. For every positive integer~$n$, let $x_n$ be such that~$\varphi^n(x_n) = \varphi(x)$. Then $\varphi^{n-1}(x_n)-x$ belong to $Ker(\varphi)$ and, therefore, belongs to $Im(\varphi^{n-1})$. Thus $x$ belongs to~$Im(\varphi^{n-1})$. The claim follows, and we can apply Lemma~\ref{lemfac1} to conclude. 
\end{proof}
One may wonder whether or not the equality~$\Atrac(T) = \fvie{T}$ holds for every linear continuous self-map with dense range on a countably infinite-dimensional Hilbert space. Actually, it is not the case, even if it is not so easy to obtain a counterexample (see Proposition~\ref{prop:contrehilbert}). We build it up using Example~\ref{lem:infdim2}. 
\begin{Lem} We define a bijection~$\alpha : (\mathbb{N}\setminus\{0\})^2\to \mathbb{N}\setminus\{0\}$  by setting $$\alpha(k,n) = \frac{k(k+1)}{2}+ \frac{(n-1)(2k+n-2)}{2}.$$  
\end{Lem}
\begin{proof}
$$\begin{array}{|c|c|c|c|c|c|c|}\hline &n=1&n=2&n=3&n=4&\cdots&n\\\hline k = 1&1&2&4&7&&\alpha(1,n)\\\hline k= 2&3&5&8&12&&\alpha(2,n)\\\hline k = 3&6&9&13&18&&\alpha(3,n)\\\hline\vdots&&&&&&\vdots\\\hline k&\alpha(k,1)&\alpha(k,2)&\alpha(k,3)&\alpha(k,4)&\cdots&\alpha(k,n)\\\hline
  \end{array}$$
\end{proof}
Let~$H$ be the Hilbert space~$\ell^2(\mathbb{C})$ and fix an Hilbert base~$(e_k)_{k\geq 0}$. We define the linear map $\hat{T} : H\to  H$ by setting $$\left\{\begin{array}{lcll}\hat{T}(e_0)&=&0 ;\\\hat{T}(e_{\alpha(k,1)})&=&\frac{1}{k^2}e_0 -\sum_{n\geq 1}\frac{1}{(n+1)\cdots(n+k+1)}e_{\alpha(k,n)}&k\geq 1;\\\hat{T}(e_{\alpha(k,n+1)})&=&\frac{1}{n+1}e_{\alpha(k,n)}&k,n\geq 1. \end{array}\right.$$

We are going to prove that
\begin{Prop}\label{prop:contrehilbert} The map~$\hat{T}$ is a linear continuous self-map with dense range, and $\Atrac(\hat{T})\neq \fvie{\hat{T}}$.  
\end{Prop}
The reader may note that for each $k$ the restriction of $\hat{T}$ to the Hilbert subspace of $H$ generated by the set~$\{e_{\alpha(k,n)}\mid n\geq 1\}\cup \{e_0\}$ is closed to the definition of the map~$T:H\to H$ of Example~\ref{lem:infdim2}. In the sequel, we denote by $f_k$ the element~$\sum a_ne_n$ of~$H$ defined by~$a_{\alpha(k,n)} = \frac{1}{(n+1)\cdots(n+k)}$,  $a_{\alpha(j,n)} = 0$ for $j\neq k$ and~$a_0 = 0$.

\begin{proof} The map~$\hat{T}$ is linear by definition. It is continuous since~$\||\hat{T}\|| \leq \frac{\pi^2}{6} + \frac{\sqrt{6}\pi}{3}$. Indeed, $\|\hat{T}(\sum a_ne_n)\| \leq \sum_{k = 1}^{+\infty} \frac{|a_{\alpha(k,1)}|}{k^2} + \sqrt{\sum_{k = 1}^{+\infty} |a_{\alpha(k,1)}|^2\sum_{n = 1}^{+\infty}\frac{1}{(n+1)^2\cdots (n+k+1)^2}} +\sqrt{\sum_{n = 1}^{+\infty}\frac{1}{(n+1)^2}\sum_{k = 1}^{+\infty}|a_{\alpha(k,n+1)}|^2}$. It has dense range because all the $e_i$ belong to the image of~$\hat{T}$. Finally, the last part of the proposition is a direct consequence of the two following lemmas.
 
\end{proof}

\begin{Lem}\label{lemcalculker}
Let~$g = \sum a_ne_n$ belong to $H$. Then,\\(i)  $\hat{T}(g) = 0$ if and only if $g = a_0e_0 + \sum_{k = 1}^{+\infty}a_{\alpha(k,1)}f_k$ with $\sum_{k = 1}^{+\infty}\frac{a_{\alpha(k,1)}}{k^2} = 0$.\\(ii) $\hat{T}(g) = e_0$ if and only if $g = a_0e_0 + \sum_{k = 1}^{+\infty}a_{\alpha(k,1)}f_k$ with $\sum_{k = 1}^{+\infty}\frac{a_{\alpha(k,1)}}{k^2} = 1$. 
\end{Lem}
\begin{proof}
This is a direct computation using the definition of~$\hat{T}$. 
\end{proof}

\begin{Lem} Let~$g$ belong to~$H$ such that $\hat{T}(g) = e_0$. Write $g = a_0e_0 + \sum_{k = m}^{+\infty}\lambda_k f_k$ with $\sum_{k = m}^{+\infty}\frac{\lambda_k}{k^2} = 1$ and $\lambda_m\neq 0$. Then $g$ belongs to $Im({\hat{T}}^{m-1})$ but not to~$Im({\hat{T}}^{m})$.
\end{Lem}
\begin{proof} For $k\geq j\geq 1$, let~$f_{k,j} = \sum a_ne_n$ be such that~$a_0 = 0$, $a_{\alpha(\ell,n)} = 0$ for~$\ell\neq k$ or $1\leq n\leq j-1$, and $a_{\alpha(k,n)} = \frac{1}{(n+1)\cdots(n+k-j+1)}$ when~$n\geq j$. One has~$f_{k,1} = f_k$ for every~$k$. By definition of the map~$\hat{T}$, we have~$\hat{T}(f_{k,j}) = f_{k,j-1}$ for $2\leq j\leq k$. Furthermore,~$f_{k,k}$ does not belong to $Im(\hat{T})$: Assume it is the case and that $\hat{T}(\sum  a_ne_n) = f_{k,k}$. As in Lemma~\ref{lemcalculker}, a computation proves that $a_{\alpha(k,n)} = 1+\frac{a_{\alpha(k,1)}}{(n+1)\cdots (n+k)}$ for every~$n\geq k+1$. This is impossible, since $\sum a_ne_n$ lies in~$H$. Let $m$ be a positive integer, and assume $h$ lies in $H$ with~$\hat{T}^m(h) = e_0$. We claim that\\ \begin{tabular}{ll}(a)&$h = a_0 e_0 + \sum_{r = 1}^m \sum_{k = r}^{+\infty}\lambda_{k,r} f_{k,r}$ with $\lambda_{k,r}\in\mathbb{C}$ with $\sum_{k = 1}^{+\infty}\frac{\lambda_{k,m}}{k^2} = 1$ and\\&$\sum_{k = 1}^{+\infty}\frac{\lambda_{k,r}}{k^2} = 0$ for~$1\leq r < m$.\\(b)&$h$ belongs to~$Im(\hat{T})$ if and only if $\lambda_{r,r} = 0$ for every $r$ in $\{1,\cdots, m\}$.\end{tabular}\\ For $r$ in $\{1,\cdots,m\}$, this implies~$\hat{T}^r(h) =  a_r e_0 + \sum_{j = 1}^{m-r} \sum_{k = j}^{+\infty}\lambda_{k,j+r} f_{k,j}$ where $a_r = \sum_{k = 1}^{+\infty}\frac{a_{\lambda_{k,r}}}{k^2}$. Indeed, for $m = 1$, the Part~(a) of the claim is equivalent to Lemma~\ref{lemcalculker}(ii); Part~(b) is a consequence of the above computation: $e_0$ and all the $f_{k,1}$ belong to $Im(\hat{T})$ except $f_{1,1}$. We deduce the claim for $m\geq 2$ by an easy induction: Part~(a) follows from equalities~$\hat{T}(f_{k,j}) = f_{k,j-1}$ for $j\leq k$ and Lemma~\ref{lemcalculker}(i), Part~(b) follows from similar arguments to the case $m = 1$. Now, let $g$ be as in the lemma, and set $f_{0,1} = e_0$. The element~$g$ belongs to $Im({\hat{T}}^{m-1})$ because~$\hat{T}^{m-1}(m^2a_0f_{m-1,m} + \sum_{k = m}^{+\infty}\lambda_k f_{k,m}) = g$. Moreover, if $\hat{T}^{m-1}(h) = g$, then $h = a e_0 + \sum_{j = 1}^m \sum_{k = j}^{+\infty}\lambda_{k,j} f_{k,j}$ with $\lambda_{m,m} = \lambda_m\neq 0$, then~$h$ does not belong to~$Im({\hat{T}})$.\end{proof}
\section{Unit open disc and analytic function}
\label{secteanalyt}
As explained at the end of Section~\ref{section:example}, the stable set is deeply connected with the inverse limit. The latter has been investigated in the case of an analytic self-map over the unit open disc~\cite{PoC2,PoC3}. Here we focus on this particular case.\begin{Prop}\cite{PoC1}
Consider an analytic function~$\varphi$ on the open unit disc~$\mathbb{D}$ of the complex plane such that $|\varphi(z)|<1$. Then, the set~$\Atrac(\varphi)$ and~$\fvie{\varphi}$ may be distinct.\end{Prop}
\begin{proof} The following example has been provided to us by P. Poggi-Corradini. It is closed to Example~\ref{lem:contreexemple}.
Let $H = \{z\in\mathbb{C}\mid Re(z)>0\} \setminus \{2^n\mid n\in\mathbb{N}\}$, and let~$f :H\to H$ be the multiplication by
$\frac{1}{2}$.  Let~$\tilde{S}$ be the universal covering of $H$, and choose $\tilde{x}_0\in \tilde{S}$ that projects down to~$x_0 = \frac{3}{4}$. The self-map~$f$ lifts to an analytic self-map~$\tilde{\varphi}$ of $\tilde{S}$. The pre-image of~$\tilde{x}_0$ under~$\tilde{\varphi}$ is in one-to-one correspondence with all the equivalent classes of paths in $H$ that start at $x_0$ and end at $x_1 = 3/2$. Label the pre-images by~$(x_{1,j})$. Now for
every $j$ consider the set of pre-images under $(1/2)^j$ and call them $(a_{k,j})$.
These are points that project from $\tilde{S}$ down to $3\cdot 2^{j-1}$. Draw a slit of the form
$\{(3*2^{j-1},y)\mid  y\geq 0\}$ on each sheet of $\tilde{S}$ containing one of the $a_{k,j}'$s. Draw a similar vertical half-slit from every pre image of the $a_{k,j}$'s. Call the obtained surface $S$. The map~$\tilde{\varphi}$ restrict to a self-map of $S$, and by the \emph{Uniformization Theorem} \cite[Chap.~IV]{Rey},  the surface $S$ is conformally equivalent to~$\mathbb{D}$. Therefore, the restriction of~$\tilde{\varphi}$ gives
rise to an analytic self-map~$\varphi$ of $\mathbb{D}$. Now, by construction, the point of~$\mathbb{D}$ corresponding to~$\tilde{x}_0$ belongs to~$\Atrac(\varphi)$, but not to~$\fvie{\varphi}$. 
\end{proof}
\section{Compact metric spaces and continuous functions}
Inverse limit of continuous function on a compact metric space has also been considered~\cite{BaRo}. Here we focus on this particular case. \begin{Prop} \label{prop:compact}Let $X$ be a compact metric space, and consider a continuous function~$\varphi:X\to X$. Then~$\Atrac(\varphi)$ is a non-empty compact subset and $$\fvie{\varphi} = \Atrac(\varphi).$$
\end{Prop}
Clearly, $\Orb(\varphi)$ is not equal to $\fvie{\varphi}$ in general (consider $\varphi: [0,1]\to [0,1]$ defined by $\varphi(x) = \frac{1}{2}$ for $x\in[0,\frac{1}{2}]$ and $\varphi(x) = \frac{3}{2}(x-\frac{1}{2})$ otherwise). 
\begin{proof} 
The set~$\Atrac(\varphi)$ is a non-empty compact subset as an intersection of nested compact subsets of $X$. Assume $x$ belongs to $\Atrac(\varphi)$, and for each~$n$ consider a finite sequence~$x_{n,0},\cdots, x_{n,n}$ such that $x_{n,0} = x$ and $\varphi(x_{n,m}) = x_{n,m-1}$ for~$m\in \{1,\cdots,n\}$. Since $X$ is compact, each infinite sequence~$(x_{n,m})_{n}$ possesses a sub-sequence that converge to an element $y_m$ of $X$. Clearly we have $y_0 = x$ and $\varphi(y_{m+1}) = y_m$ for every $m$. Thus $x$ belongs to~$\fvie{\varphi}$.\end{proof}
\label{sectcompa}
\section{Free groups and homomorphims}
We recall that every subgroup of a free group is a free group~\cite[Cor.~2.9]{MaKaSo} and  that finite rank free groups are hopfian, that is every surjective self-homomorphism is indeed an automorphism~\cite[Th.~2.13]{MaKaSo}. 
\begin{Prop} Let $X = F_m$ be the free group of rank~$m$, and consider a group endomorphism~$\varphi: F_m\to F_m$. Then, $$\fvie{\varphi} = \Atrac(\varphi).$$ \label{prop:gplibre}
\end{Prop}
It is easy to see  that the set $\Orb(\varphi)$ is not equal to $\fvie{\varphi}$ in general (just consider some inner automorphism of $F_n$). 
\begin{proof}The main argument is like in~\cite[Theorem.~1]{ImrTur}: the sequence~$rk(\varphi^n(F_m))$ eventually stabilizes, and the restrictions $\varphi: \varphi^n(F_m)\to \varphi^{n+1}(F_{m})$ is surjective. Since free groups are hopfian, it follows that the restriction of $\varphi$ to~$\varphi^n(F_m)$ is injective for sufficiently large~$n$. Then, we apply Lemma~\ref{lemfac2}.
\end{proof}
The above result can be extended without difficulty to \emph{limit groups of free groups}. Let us briefly recall how they are defined~\cite{Gui}. Fix a positive integer~$n$. A \emph{marked group} is a pair $(G,S)$ where~$G$ is a group and $S: F_n\to G$ is an homomorphism  that is onto. The \emph{space of marked groups}~$\mathcal{G}_n$ is the set of marked groups where two marked groups~$(G,S)$ and $(G',S')$ are identified when there exists a group isomorphism~$\varphi : G\to G'$ such that $S' = \varphi\circ S$. The set~$\mathcal{G}_n$ can be identified with a subset of the set~$\{0,1\}^{F_n}$ of subsets of $F_n$ (an equivalence class of pairs~$(G,S)$ is identified with the kernel of $S$). The product topology on~$\{0,1\}^{F_n}$ makes~$\mathcal{G}_n$ a compact space. A \emph{limit group of free groups} is a group $G$ such there exists a pair $(G,S)$ in the closure of the free groups in~$\mathcal{G}_n$ (that is in the closure of the set of pairs~$(F_m,S)$).
\begin{Prop} Let $X$ be a limit group of free groups and  consider a group endomorphism~$\varphi: X\to X$. Then, $$\fvie{\varphi} = \Atrac(\varphi).$$  \end{Prop}
\begin{proof}
The main argument is like in Proposition~\ref{prop:gplibre} using~\cite[Prop.~3.13]{Gui} (see also~\cite[Prop.~5.1]{Sel}) to conclude.
\end{proof}
\label{sectfree}
\section{Finite words, infinite words and substitutions}
\label{sectsubsti} Let $S$ be a finite set.  We denote by $S^*$ and $S^\omega$ the free monoid of finite words and the set of right infinite words, respectively, on the alphabet~$S$. We denote by~$\varepsilon$ the trivial word. A \emph{substitution} is a map from $S$ to~$S^*$. It is clear that a substitution~$\varphi: S\to S^*$ extends to an endomorphism of monoid~$\varphi_*: S^*\to S^*$ defined by $\varphi_*(s_1\cdots s_n) = \varphi(s_1)\cdots\varphi(s_n)$. We say that a letter $s$ in $S$ is \emph{mortal} (\emph{for the substitution}) if there exists a positive integer~$n$ such that $\varphi_*^n(s)$ is the empty word. The set of mortal letters is denoted by~$M_\varphi$. When $s$ is not mortal, we say that it is \emph{immortal}. The map~$\varphi$ is said to be \emph{non-erasing} if the set~$M_\varphi$ is empty. When~$\varphi$ is non-erasing, it can also be extended to a map~$\varphi_\omega: S^\omega\to S^\omega$ defined by $\varphi_\omega(s_1s_2\cdots) = \varphi(s_1)\varphi(s_2)\cdots$. A special case occurs when~$S = \{s_1,s_2\}$ and~$\varphi$ extends to an automorphism of~$F_2$. The elements of~$\fvie{\varphi_\omega}$ are special kinds of the so-called~\emph{sturmian words}~\cite{GlLeRi}. In the erasing case, we can still extend~$\varphi$ to a map a map from~$S^\omega$ to $S^\omega\cup S^*$ that restricts to a self-map~$\varphi_\omega: S_\infty^\omega\to S_\infty^\omega$, where~$S_\infty^\omega$ is the set of infinite words in~$S^\omega$ that contains an infinite number of immortal letters. When no confusion is possible, we will write~$\varphi$ for~$\varphi_*$ and $\varphi_\omega$. 

\begin{Prop} \label{prop:finworderasornot}Let $S$ be a finite set and  $\varphi : S\to S^*$ be a substitution. Then,\\ (i) $\Orb(\varphi_*) = \fvie{\varphi_*} = \Atrac(\varphi_*)$;\\
(ii) $\Orb(\varphi_\omega) = \fvie{\varphi_\omega} = \Atrac(\varphi_\omega)$;\\
(iii) When $\varphi$ is non-erasing, then $\Orb(\varphi_*)$ is the free monoid~$\{s\in S\mid \varphi(s)\in S\}^*$. 
\end{Prop}

\begin{Rem} (i) In Proposition~\ref{prop:finworderasornot}, Point~$(i)$ is a consequence of Point~$(ii)$. Indeed, consider an extra letter~$t$ and the set $T = S\cup\{t\}$. 
Define the map~$\tilde{\varphi}:T\to T^*$ by $\tilde{\varphi}(s) = \varphi(s)$ for $s$ in $S$ and $\tilde{\varphi}(t) = t$. Let $\psi: S^*\to T_\infty^\omega$ be the map defined by $\psi(w) = wttt\cdots$.
It is immediate that~$W$ belongs to~$\Orb(\varphi_*)$, $\fvie{\varphi_*}$ and $\Atrac(\varphi_*)$ if and only if~$\psi(W)$ belongs to~$\Orb(\tilde{\varphi}_\omega)$, $\fvie{\tilde{\varphi}_\omega}$ and~$\Atrac(\tilde{\varphi}_\omega)$, respectively.\\
(ii) It is well-known that $S^\omega$ is compact for its standard metric defined~\cite[Cor.~3.13]{PePi}. A non-erasing substitution is continuous. By Proposition~\ref{prop:compact}, the sets~$\fvie{\varphi_\omega}$ and $\Atrac(\varphi_\omega)$ are equal in this case. 
 \end{Rem}

To prove Proposition~\ref{prop:finworderasornot} we need to introduce some notation. Consider a substitution~$\varphi : S\to S^*$. The~\emph{mortality exponent}~$exp(\varphi_*)$ of $\varphi_*$ is the least integer~$m$ such that~$\varphi_*^m(s)$ is the empty word for every letter in~$M_\varphi$. One has $exp(\varphi_*) \leq Card(M_\varphi)\leq Card(S)$. When $W$ belongs to~$S^*$, we denote  by $\ell(W)$ and by~$\ell_\infty(W)$ its length and its number of immortal letters, respectively. We will need the following classical lemma~\cite{ShWa}.

\begin{Lem} Let~$a$ lie in $S$ such that $\varphi(a) = V_1aV_2$ with $\ell_\infty(V_1) = 0$.\\
(i) If $\ell_\infty(V_2) = 0$, then $\varphi_*^{exp(\varphi_*)}(a)$ is fixed by $\varphi_*$.\\
(ii) If $\ell_\infty(V_2) \neq 0$ then for every $n$ one has $\ell(\varphi_*^n(a))\geq n+1$, and there exists an infinite word~$\overrightarrow{\varphi_*^\omega}(a)$ which is fixed by~$\varphi_\omega$ and such that for every~$n\geq exp(\varphi_*)$ the finite word~$\varphi_*^n(a)$ is one of its (left) prefixes.\label{Lem:lemptfixe} 
\end{Lem}

\begin{proof}Fix an integer~$m\geq q = exp(\varphi_*)$. Then, one has~$\varphi_*^{m}(V_1) = \varepsilon$.  If $\ell_\infty(V_2) = 0$ then $\varphi_*^{m}(V_2) = \varepsilon$. Otherwise,~$\ell(V_2\varphi_*(V_2)\cdots \varphi_*^{m}(V_2)) \geq \ell_\infty(V_2\varphi_*(V_2)\cdots \varphi_*^{m}(V_2)) \geq m$.  But~$\varphi_*^m(a)$ is equal to~$\varphi_*^{m-1}(V_1)\cdots V_1aV_2\cdots \varphi_*^{m-1}(V_2)$. Hence, if $\ell_\infty(V_2) = 0$, then it follows that~$\varphi_*^{q+1}(a) = \varphi_*^{q}(a) = \varphi_*^{q-1}(V_1)\cdots V_1aV_2\cdots \varphi_*^{q-1}(V_2) = \varphi_*^{q}(a)$. If we assume~$\ell_\infty(V_2) \neq 0$, then $\varphi_*^{q-1}(V_1)\cdots V_1aV_2\varphi_*(V_2)\varphi^2_*(V_2)\cdots$ belongs to~$S_\infty^\omega$ and is fixed by $\varphi_\omega$; furthermore~$\varphi_*^{m}(a)$ is one of its left prefix.
 \end{proof}

We are now ready to prove Proposition~\ref{prop:finworderasornot}.

\begin{proof}[Proof of Proposition~\ref{prop:finworderasornot}]
We start with a direct proof of~$(i)$ in the non-erasing case. This proves~$(iii)$ too. Let $W$ belong to $\Atrac(\varphi_*)$. Clearly, for every word~$V$ in~$S^*$, the equality~$\ell(\varphi(V))\geq \ell(V)$ holds. Now, choose $N\geq (\textrm{Card}(S))^{\ell(W)}$. Consider $W_N$ such that $\varphi^N(W_N) = W$. Since~$\ell(\varphi^j(W_N))\leq\ell(W)$ for every $j$ in~$\{0,\ldots,N\}$, there exist $j < i$ in $\{0,\ldots,N\}$ such that $\varphi^j(W_N) = \varphi^i(W_N)$. Hence, $\varphi^j(W_N)$ and $W$ belong to $\Orb(\varphi)$. 

We turn now to the proof of~$(ii)$. It is in the same spirit than the above one, but more technical. Let~$W$ belong to~$\Atrac(\varphi_\omega)$ and assume it does not belong to~$\Orb(\varphi_\omega)$. Denote by $k$ the number of immortal letters of~$\varphi$, that is~$k = Card(S)-Card(M_\varphi)$. Write $W = UV$ where $U$ is the maximal periodic prefix ( for~$\varphi_*$) of $W$ with period equal or lower than~$k$. The word $U$ exists since $W$ is not periodic and $k$ is fixed. The word~$V$ lies also in~$\Atrac(\varphi_\omega)$ and not in~$\Orb(\varphi_\omega)$. Thus, we can assume without restriction that~$U$ is the empty word, and $W = V$. We choose a sequence $(V_n)$ in $S_\infty^\omega$ such that $\varphi^n_\omega(V_n) = W$ for each $n$.  We set $W_{n,i} = \varphi_\omega^{n-i}(V_n)$ for $0\leq i\leq n$. In particular, $W_{n,n} = V_n$ and $W_{n,0} = W$. Let $s_{n,i}$ denote the first (left) immortal letter of~$W_{n,i}$. For each $n$, there exists~$s_n$ in $S$ such that~$Card\{0\leq i\leq n\mid s_{n,i} = s_n\} \geq \lambda_n = E(\frac{n-1}{k})+1$. Since $S$ is finite, there exists an immortal letter~$s$ and~an increasing map~$\psi: \mathbb{N}\to \mathbb{N}$ such that $s_{\psi(n)} = s$ for every~$n$. By hypothesis, each $s_{n,i}$ is the first immortal letter of the word~$W_{n,i}$. It follows there exists a minimal positive integer~$r$ such that $\varphi^r(s) = V_1sV_2$ with~$\ell_\infty(V_1) = 0$. We must have  $1\leq r\leq k$. Furthermore, for every $n$ there exists~$i_n$ in~$\{0,\cdots, k-1\}$ such that $s_{\psi(n),i_n+jr} = s$ for all~$j$ in $\{0,\cdots, \lambda_{\psi(n)}-1\}$. We deduce there exists~$\iota$ in~$\{0,\cdots, k-1\}$ and $\psi_1: \mathbb{N}\to \mathbb{N}$ with $s_{\psi_1(n),\iota+jr} = s$ for every $n$ in~$\mathbb{N}$ and  every~$j$ in~$\{0,\cdots, \lambda_{\psi_1(n)}-1\}$. Assume~$\ell_\infty(V_2) = 0$, and consider $n$ such that $\lambda_{\psi_1(n)}\geq exp(\varphi_*) \geq exp(\varphi_*^r)$. By Lemma~\ref{Lem:lemptfixe}, $\varphi_*^i(s)$ is fixed by~$\varphi_*^r$ for $i \geq  exp(\varphi_*^r) r$. Therefore, $\varphi_*^{\lambda_{\psi_1(n)}r+\iota}(s)$ is fixed by~$\varphi_*^r$. But the word~$\varphi_*^{\lambda_{\psi_1(n)}r+\iota}(s)$ is a  non-trivial prefix of $\varphi_\omega(W_{\psi_1(n), \lambda_{\psi_1(n)}r+\iota})$, that is of~$W$ : a contradiction since $r\leq k$. Assume~$\ell_\infty(V_2) \neq 0$.  For every $n$ such that $\lambda_{\psi_1(n)}\geq exp(\varphi_*) \geq exp(\varphi_*^r)$, the word~$\varphi_*^\iota(\varphi_*^{\lambda_{\psi_1(n)}r}(s))$ is a prefix of $W$ and its length is greater than $\lambda_{\psi_1(n)}$. Thus, by Lemma~\ref{Lem:lemptfixe}, the word~$W$ is equal to~$\varphi_\omega^\iota\left(\overrightarrow{(\varphi_*^r)^\omega}(a)\right)$ and is therefore fixed by $\varphi^r_*$, again a contradiction.      
\end{proof}
\begin{Rem} Consider the hypotheses of Proposition~\ref{prop:finworderasornot}. For an immortal letter~$s$ in $S$, denote by $m_s$ the minimal positive integer such that the first left immortal letter of $\varphi^{m_s}(s)$ is~$s$, when it exists. Let $m$ be the lcm of the $m_s$. What we actually show when proving  Proposition~\ref{prop:finworderasornot} is that $\fvie{\varphi_\omega} = \Fix(\varphi_\omega^{m})$. \end{Rem}

\begin{Exe} Let $S = \{a,b\}$. Define $\varphi_\omega: \{a,b\}^\omega\to\{a,b\}^\omega$ by $\varphi(a) = ab$ and $\varphi(b) = ba$. This substitution is non-erasing, and is called the~\emph{Thue-Morse} substitution. It turns out that~$\varphi_\omega$ has a unique fixed point~$W = abba\cdots$, the well-known \emph{Thue-Morse infinite word}. By the above remark,~$\Fix(\varphi_\omega)=\Orb(\varphi_\omega) = \fvie{\varphi_\omega} = \Atrac(\varphi_\omega) = \{W\}$. 
 \end{Exe}

\section{The stable set of a monoid of self-maps}
\label{sectgenera}
The original motivation for this note is to have a better understanding of the connection between a monoid of self-maps and what we call below its stable set. More precisely we are interested by the stable set of infinite words associated with a given finite family of substitutions. The seminal case we have in mind is the one of episturmian morphisms and its stable set, which is the set of \emph{infinite epistumian words}\cite{GlLeRi}. The object of this last section is to introduce the notion of a \emph{stable set of a monoid of self-maps} and to provide motivating examples.

 In the sequel, we fix a non-empty set $X$, and denote by $F(X)$ the set of \emph{self-maps} of $X$. If $\smap$ is included in~$F(X)$, we denote by $\wsmap$ and by $\msmap$ the free monoid on the alphabet~$\smap$ and the sub-monoid of~$F(X)$ generated by~$\smap$, respectively.

The notion of an inverse limit was recalled in Section~\ref{section:example}. We remark that for any non-empty subset~$\smap$ of $F(X)$, the free monoid~$\wsmap$ is a poset for the prefix order: $\varphi_1\cdots \varphi_n \leq \varphi'_1\cdots \varphi'_m$ if $n\leq m$ and $\varphi_i = \varphi'_i$ for $1\leq i\leq n$. 
\begin{definition}
Consider a non-empty subset~$\smap$ of $F(X)$. Denote by $f\to\overline{f}$ the canonical surjective morphism from $\wsmap$ to $\msmap$.  If~$\varphi\leq\varphi'$ in $\wsmap$, we denote by~$\psi_{\varphi,\varphi'}$ the element of~$\msmap$ such that $\overline{\varphi'} = \overline{\varphi}\circ \psi_{\varphi,\varphi'}$.  Let~$(X_\varphi, \psi_{\varphi,\varphi'})$ be the projective system defined by $X_\varphi = X$ for $\varphi\in \wsmap$.  The \emph{stable set}~$Stab(\smap)$ of~$\smap$ is defined by  $$Stab(\smap) = \pi_0(\underleftarrow{\lim}(X_\varphi))$$ 
\end{definition}

It is immediate that the above definition generalized the definition of the stable set of a self-map. When~$\smap = \{\varphi\}$, the sets $Stab(\{\varphi\})$ and~$\fvie{\varphi}$ are equal.  As far as we know, this notion of stable set has not been considered before. However, interesting sets of infinite words occur as stable sets:  

\begin{Exe} Let~$\Sigma$ be a finite set and denote by $\Sigma^\omega$ the set of left infinite words over~$\Sigma$. For $a$ in~$\Sigma$, denote by $L_a: \Sigma^\omega\to \Sigma^\omega$ and $R_a: \Sigma^\omega\to \Sigma^\omega$ the substitutions defined by $L_a(b) = ab$, $R_a(b) = ba$ for $b\neq a$ in~$\Sigma$ and  $L_a(a) = R_a(a) = a$. Set $\smap = \{L_{a}, R_{a}, a\in \Sigma\}$. The elements of the monoid generated by $\mathcal{F}$ are called \emph{episturmian morphisms}~\cite{{GlLeRi}}. The elements of $Stab(\smap)$ are the so-called~\emph{infinite episturmian words}~\cite[Theorem~3.10]{JuPi}, that have been introduced independently of  (and before) the episturmian morphisms~\cite{DrJuPi}.   
 
\end{Exe}
\begin{Exe}Let $\Sigma = \{1,2,\cdots,n\}$ and denote by $\Sigma^\omega$ the set of left infinite words over~$\Sigma$. We say that a word~$W$ is \emph{2-square free} if any two consecutive letters are distinct.  For each 2-square free word~$W = a_1a_2a_3\cdots$ in~$\Sigma^\omega$, we define the map~$\psi_W : \Sigma^\omega\to\Sigma^\omega$ by $$\psi_W(x_1x_2x_3\cdots) = \underbrace{a_1\cdots a_1}_{x_1}\underbrace{a_2\cdots a_2}_{x_2}\underbrace{a_3\cdots a_3}_{x_3}\cdots $$
Given a word~$V$ in~$\Sigma^\omega$, there is at most one pair~$(W,\Delta(V))$ in~$(\Sigma^\omega)^2$ that verifies~$\psi_W(\Delta(V)) = V$. The map $\Delta$ is the well-known \emph{run-length encoding} used in compressing data algorithms~\cite{BrJaPa}. The so-called \emph{Kolakoski word}~$2211212211\cdots$ is the unique fixed point of~$\psi_{21212\cdots}$ in $\{1,2\}^\omega$. If for~$\mathcal{F}$ we consider the set of maps~$\psi_W$, where the words~$W$ are 2-square free in $\Sigma^\omega$, then the elements of $Stab(\smap)$ are by definition the so-called \emph{smooth words over~$\Sigma$} considered in~\cite{BeBrCh,BrJaPa}. \label{exekola}
\end{Exe}

The basic property of stable-sets of self-maps stated in Lemma~\ref{lemfac1} extends without difficulties to the wider context of stable sets of self-map monoids:
\begin{definition}
Let $\smap$ be included in~$F(X)$ and~$Y$ be a subset of~$X$. We say that~$Y$ is \emph{stabilized} by~$\smap$ if $$\bigcup_{\varphi\in\smap}\varphi(Y) = Y.$$ \end{definition} 

\begin{Lem} The set~$Stab(\smap)$ is stabilized by~$\smap$, and every set stabilized by~$\smap$ is included in $Stab(\smap)$.
\end{Lem}
\begin{proof} The proof is immediate.
\end{proof}

Let us finish this note with a question and a comment.

\begin{Que} Consider a non-empty subset~$\smap$ of $F(X)$, and denote by $f\to\overline{f}$ the canonical surjective morphism from $\wsmap$ to $\msmap$. Defined the \emph{attracting set} of~$\smap$ to be $$\Atrac(\smap) = \bigcap_{n\geq 1}\left(\bigcup_{f\in \smap^*, \ell(f) = n} \overline{f}(X)\right)$$ As in the case of a single self-map, it is clear that~$Stab(\smap)$ is included in~$\Atrac(\smap)$. When is this inclusion an equality ?   
\end{Que}

We remark that in the case of Example~\ref{exekola}, it is trivial that the equality holds by run-length encoding map. In general, answering to this question seems more complicated. The reader can note that when $\smap$ is finite, then for every element~$x_0$ in~$\Atrac(\smap)$ the exists a right infinite word~$f_1f_2\cdots$ in~$\smap^\omega$ and a sequence~$(x_n)$ in $X$ such that for every~$n$ one has $x_0 = f_0(f_1(\cdots f_n(x_n)))$. However, this does not mean that we can choose $x_{n-1} = f_n(x_n)$.  

\noindent {\bf Acknowledgments.} The author is grateful to F. Lev\'e, G. Levitt, P. Poggi-Corradini, G. Richomme and  L. Vainerman for useful and motivating exchanges. 



\begin{thebibliography}{10}
\bibitem{Arnetall}
{\sc Arnoux P., Berth\'e V., Hilion A. and Siegel A.}
\newblock Fractal representation of the attractive lamination of an automorphism of the free group.
\newblock{\em Ann. Inst. Fourier \/} 56 (2006) 2161--2212.

\bibitem{BaRo}
{\sc Barge M. and Roe R.}
\newblock Circle maps and inverse limits.
\newblock {\em Topology and Appl.\/} 36 (1990) 19--26.

\bibitem{BeHa}
{\sc Bestvina M. and Handel M.}
\newblock Train tracks for surface homeomorphisms.
\newblock {\em Topology.\/} 34 (1995) 109--140.

\bibitem{BeBrCh}
{\sc Berth\'e V., Brlek S. and Choquette P.}
\newblock Smooth words over arbitrary alphabets.
\newblock {\em Theoret. Comput. Sci.\/} 341 (2005) 293--310.

\bibitem{BrJaPa}
{\sc Brelk S., Jamet D. and Paquin G.}
\newblock Smooth words on 2-letter alphabets having same parity.
\newblock {\em Theoret. Comput. Sci.\/} 393 (2008) 166--181.

\bibitem{Gui}
{\sc Champetier C. and Guirardel V.}
\newblock Limit groups as limits of free groups.
\newblock {\em Israel J. Math.\/}  146 (2005) 1--75.

\bibitem{ChEvWa}
{\sc Chothi, V., Everest, G. and Ward T.}
\newblock $S$-integer dynamical systems: periodic points.
\newblock {\em J. Reine Angew. Math.\/} 489 (1997) 99--132.

\bibitem{CuKaLe}
{\sc Culik K., Karhum\"aki J. and Lepist\"o A.}
\newblock Alternating iteration of morphisms and the Kolakovski sequence.
\newblock {\em Lindenmayer systems\/} 93--106, Springer, Berlin, 1992 .

\bibitem{DrJuPi}
{\sc Droubay X., Justin J. and Pirillo G.}
\newblock Episturmian words and some constructions of de Luca and Rausy.
\newblock {\em Theoret. Comput. Sci.\/} 255 (2001) 539--553.


\bibitem{DySc}
{\sc Dyer J.L. and Scott G.P.}
\newblock Periodic automorphisms of free groups.
\newblock {\em Comm. Alg.\/} 4 (1975) 195--201.

\bibitem{Est}
{\sc Esterle J.}
\newblock Elements for classification of commutative radical banach algebra.
\newblock {\em Proceddings of the 1981 Long Beach Conference on Radical Banach Algebras and Automatic continuity.\/} Springer Verlag, Berlin.

\bibitem{FaJo}
{\sc Farrell F. T. and Jones L. E.}
\newblock New attractors in hyperbolic dynamics. 
\newblock {\em J. Differential Geom.\/} 15 (1980) 107--133.

\bibitem{GlLeRi}
{\sc Glen A., Lev\'e, F. and Richomme G.}
\newblock Quasiperiodic and Lyndon episturmian words. 
\newblock {\em Theoret. Comput. Sci.\/} 409 (2008) 578--600.

\bibitem{Hol}
{\sc Holte S.}
\newblock Embedding inverse limits of neary Markov interval maps as attracting sets of planar diffeomorphisms.
\newblock {\em Colloquium mathematicum.\/} 68 (1995) 291--296.

\bibitem{ImrTur}
{\sc Imrich W. and Turner E.C.}
\newblock Endomorphisms of free groups and their fixed points
\newblock {\em Math. Proc. Cambridge Philos. Soc.\/} 105 (1989) 421--422.

\bibitem{JaSh}
{\sc Jaco W. and Shalen P.B.}
\newblock Surface homeomorphisms and periodicity.
\newblock {\em Topology.\/} 16 (1977) 347--367.

\bibitem{JuPi}
{\sc Justin J. and Pirillo G.}
\newblock Episturmian words and episturmian morphisms.
\newblock {\em Theoret. Comput. Sci.\/} 276 (2002) 281--313.

\bibitem{MaKaSo}
{\sc Magnus W., Karrass A. and Solitar D.}
\newblock Combinatorial Group Theory.
\newblock {\em Dover publications.\/} New-York 2004.

\bibitem{MeRa}
{\sc Medio A. and Raines B.}
\newblock Backward dynamics in economics. The inverse limit approach.
\newblock {\em Journal of Economy Dynamics and Control.\/} 16 (2007) 1633--1671.

\bibitem{PePi}
{\sc Perrin D. and Pin J.-E.} 
\newblock Infinite Words.
\newblock{\em Pure and applied Mathematical series\/} 141, Elsevier (1992). 

\bibitem{PoC1}
{\sc Poggi-Corradini P.}
\newblock personnal comunication.

\bibitem{PoC2}
{\sc Poggi-Corradini P.} 
\newblock Iteration of analytic self-maps of the disk: an overview 
\newblock{\em Cubo Mathematical Journal\/}  6 (2004), 73--80. 

\bibitem{PoC3}
{\sc Poggi-Corradini P.} 
\newblock Backward-iteration sequences with bounded hyperbolic steps for analytic self-maps of the disk 
\newblock{\em Revista Matematica Iberoamericana\/} 19 (2003) 943--970. 

\bibitem{Rey}
{\sc Reyssat E.}
\newblock Quelques aspects des surfaces de Reimann.
\newblock {\em Progess in Maths Birkh\"auser.\/} Boston 1989.

\bibitem{Sel}
{\sc Sela Z.}
\newblock Diophantine geometry over groups. I. Makanin-Razborov diagrams.  
\newblock {\em Publ. Math. Inst. Hautes Études Sci. \/} 93  (2001) 31--105.

\bibitem{ShWa}
{\sc Shallit J. and Wang, M.}
\newblock On two-sided infinite fixed points of morphisms. 
\newblock {\em Theoret. Comput. Sci.\/} 270 (2002) 659--675.

\bibitem{Shp}
{\sc Shpilrain V.}
\newblock  Fixed points of endomorphisms of a free metabelian group.
\newblock{ \em Math. Proc. of the Cambridge Philos. Soc.\/} 123 (1998) 75-83.

\bibitem{Ste}
{\sc Stein S.D.}
\newblock  Some Remarks concerning the mittag-leffler inverse limit theorem.
\newblock{ \em Periodica Mathematica Hungaria.\/} 31 (1995) 63-69.

\bibitem{Ven}
{\sc Ventura E.}
\newblock Fixed subgroups in free groups: a survey.
\newblock {\em  Contemporary Math\/} 296 (2002), 231-255.

\end{thebibliography}
\end{document}